\documentclass[11pt]{amsart}
\usepackage{mathrsfs}
\usepackage{geometry}
\usepackage{graphicx}
\usepackage{amssymb}
\usepackage{epstopdf}
\usepackage{amsmath,amscd}
\usepackage{amsthm}
\usepackage{enumerate}
\usepackage{url,verbatim}
\usepackage{subfig}
\usepackage{enumitem}
\usepackage{tikz}
\usepackage{enumerate}

\usepackage[normalem]{ulem}
\usepackage{fixme}

\calclayout

\makeatletter
\newcommand\xleftrightarrow[2][]{%
  \ext@arrow 9999{\longleftrightarrowfill@}{#1}{#2}}
\newcommand\longleftrightarrowfill@{%
  \arrowfill@\leftarrow\relbar\rightarrow}
\makeatother

\RequirePackage[colorlinks,citecolor=blue,urlcolor=blue]{hyperref}
\usepackage{breakurl}
\theoremstyle{plain}
\newtheorem{theorem}{Theorem}
\newtheorem{definition}[theorem]{Definition}
\newtheorem{lemma}[theorem]{Lemma}
\newtheorem{proposition}[theorem]{Proposition}
\newtheorem{corollary}[theorem]{Corollary}
\newtheorem{example}[theorem]{Example}

\newtheorem{conjecture}[theorem]{Conjecture}

\newcommand\RR{{\mathbb R}}

\newcommand\HH{{\mathbb H}}

\newcommand\lra{\leftrightarrow}

\renewcommand\ell{l}

\newcounter{mycount}

\numberwithin{equation}{section}
\numberwithin{theorem}{section}
\numberwithin{figure}{section}

\title{Planar site percolation on semi-transitive graphs}

\author{Zhongyang Li}
\address{Department of Mathematics,
University of Connecticut,
Storrs, Connecticut 06269-3009, USA}
\email{zhongyang.li@uconn.edu}
\urladdr{\url{https://mathzhongyangli.wordpress.com}}

\begin{document}

\maketitle

\begin{abstract}Semi-transitive graphs, defined in \cite{hps98} as examples where ``uniform percolation" holds whenever $p>p_c$, are a large class of graphs more general than quasi-transitive graphs.
Let $G$ be a semi-transitive graph with one end which can be properly embedded into the plane with minimal vertex degree at least 7. We show that $p_u^{site}(G) +p_c^{site}(G_*)=1$,  where $G_*$ denotes  the matching  graph  of $G$.     This fulfils and extends an observation of Sykes and Essam in 1964 (\cite{SE64}) to semi-transitive graphs.
\end{abstract}

\section{Introduction}

Introduced by Broadbent and Hammersley in 1957 (see \cite{BH57}) to study the random spread of a fluid through a medium, percolation has been a celebrated model illustrating the phase transition, magnetization, or the spread of pandemic diseases; see \cite{grgP,HDC18} for recent accounts of the theory.

 Let $G=(V,E)$ be a graph. We write $e=\langle u,v\rangle$ for an edge with endpoints $u$ and $v$; where $u,v\in V$ and $e\in E$.
The (\emph{vertex}-)\emph{degree} of a vertex $v\in V$ is the number 
of edges incident to $v$; i.e.~edges one of whose endpoints is $v$.  We say a graph is locally finite if each vertex has finite degree.

 Assume $G=(V,E)$ is an infinite, locally finite, connected graph. 
A site percolation configuration $\omega\in \{0,1\}^{V}$  is a an assignment to each vertex in $G$ of either state 0 or state 1. A cluster in $\omega$ is a maximal connected set of vertices in which each vertex has the same state in $\omega$. A cluster may be a 0-cluster or a 1-cluster depending on the common state of vertices in the cluster.  A cluster may be finite or infinite depending on the total number of vertices in it. We say that percolation occurs in $\omega$ if there exists an infinite 1-cluster in $\omega$.

\begin{definition}\label{df11}
A planar graph $G$ is a graph that can be drawn in the plane $
\RR^2$, with vertices represented
by points and edges represented by curves, such that edges can intersect only at vertices. 

Let $\mathcal{R}\subseteq \RR^2$ be an open, connected subset of $\RR^2$.
We call a drawing of the graph into the plane a proper embedding in $\mathcal{R}$ if any compact subset $K$ in $\RR^2$ satisfying $K\subset \mathcal{R}$ intersects at most finitely many edges and vertices. 
\end{definition}

See also Section 2 of \cite{bsjams} for the definition of proper embeddings. From Definition \ref{df11} we see that if a graph $G$ can be properly embedded into $\RR^2$, then it is locally finite.
Since $\RR^2$ is homeomorphic to any simply-connected open subset of $\RR^2$, a graph can be properly embedded into $\RR^2$ if and only if it can be properly embedded into any simply-connected open subsets of $\RR^2$, in particular the unit disk (with appropriate metric this gives the hyperbolic plane $\HH^2$); see \cite{CFKP97,bsjams,ZL17, GrL20,GrZL22,GrZL221} for graphs embedded into the hyperbolic plane $\HH^2$ and statistical mechanical models on such graphs. However, in general a graph cannot be circle packed in both $\RR^2$ and $\HH^2$(\cite{HS1,HS2}).

Of particular interest is the i.i.d.~Bernoulli site percolation on a graph. In such a model, an independent Bernoulli random variable, which takes value 1 with probability $p\in [0,1]$, is associated to each vertex.  For the i.i.d.~Bernoulli site percolation, define
\begin{small}
\begin{align}
p_c^{site}(G):&=&\inf\{p\in[0,1]: \mathrm{Bernoulli}(p)\ \mathrm{site\ percolation\ on}\ G\ \mathrm{has\ an\ infinite\ 1-cluster\ a.s.} \}\label{dpc}\\
p_u^{site}(G):&=&\inf\{p\in[0,1]: \mathrm{Bernoulli}(p)\ \mathrm{site\ percolation\ on}\ G\ \mathrm{has\ a\ unique\ infinite\ 1-cluster\ a.s.} \}\label{dpu}
\end{align}
\end{small}

It follows immediately that $p_c^{site}(G)\leq p_u^{site}(G)$. If strict inequalities hold, then for certain $p$, there is a strictly positive probability that in the i.i.d.~Bernoulli($p$) percolation at least two infinite 1-clusters exist.
A number of problems related to the uniqueness and non-uniqueness of infinite percolation clusters
were formulated by Benjamini and Schramm in
their influential paper \cite{bs96}, including the following one.

\begin{conjecture}\label{c11}
Conjecture 1.6 (Conjecture 7 in \cite{bs96}). Consider site percolation on an infinite, connected,
planar graph G with minimal degree at least 7. Then, for any $p\in  (p_c^{site},1-p_c^{site})$, a.s.~there are infinitely many infinite 1-clusters in the i.i.d.~Bernoulli($p$) site percolation on $G$. Moreover, it is the case that $p_c^{site}<\frac{1}{2}$ , so the above interval is
invariably non-empty.
\end{conjecture}

A graph is called vertex-transitive (resp.\ quasi-transitive) when there is a unique orbit (at most finitely many orbits) of vertices under the action of its automorphism group. Invariant percolation processes on quasi-transitive graphs have been studied extensively, in which the symmetry property associated with the quasi-transitivity makes the analysis more convenient, and interesting techniques were developed, for example, the mass-transport principle (MTP); see \cite{BLPS1,BLPS2}. 
Conjecture \ref{c11} was proved in \cite{ZL17} when the graph is a transitive triangulation of the plane with vertex degree at least 7; and in \cite{GrZL22} when the graph $G$ is infinite, locally finite, planar, 2-connected, simple and quasi-transitive. Furthermore, it is proved in \cite{GrZL221} that if the graph $G$ is infinite, locally finite, planar, 2-connected, simple, transitive and not a planar triangulation, the at $1-p_c^{site}$ a.s.~there are infinitely many infinite 1-clusters.
Without the quasi-transitivity assumption, many existing techniques, e.g.~ergodicity of measures, MTP, etc.,  do not work any more. 

One approach to overcome this difficulty is based on the ``uniform percolation" defined in \cite{RS98}. In the absence of quasi-transitivity, the ``uniform percolation" gives ``similarities" to different vertices which make the analysis work. One important consequence of the ``uniform percolation" is the ``stability of infinite clusters'', which states that for any $0\leq p_1<p_2\leq 1$, if there is uniform percolation at level $p_1$, then every infinite 1-cluster in the i.i.d.~Bernoulli($p_2$) percolation a.s.~contains at least one infinite 1-cluster in the i.i.d.~Bernoulli($p_1$) percolation. See also Theorem 1.10 in \cite{HT21}.
One class of graphs satisfying the uniform percolation condition are the``semi-transitive graphs" defined in \cite{hps98}, which are a general class of graphs including all the quasi-transitive graphs. The main goal of this paper is to investigate the Conjecture \ref{c11} with the ``uniform percolation" assumption but without the quasi-transitive assumption. 

Here are the main results of the paper.

\begin{theorem}\label{t68}Let $G=(V,E)$ be an infinite, connected graph of uniformly bounded vertex degree. Assume that $G$ can be properly embedded into $\RR^2$ such that one of the following 2 conditions holds:
\begin{enumerate}
    \item the minimal vertex degree is at least 7; or
    \item the minimal vertex degree is at least 5; and the minimal face degree is at least 4.
\end{enumerate} 
Furthermore, Suppose that there exists $p_0\in \left[p_c^{site}(G),\frac{1}{2}\right)$, such that 
 for every $p\in \left(p_0,\frac{1}{2}\right)$ there is uniform percolation in $G$. Then for all $p\in \left(p_0,1-p_0\right)$, a.s. there are infinitely many infinite 1-clusters and infinitely many infinite 0-clusters.
\end{theorem}

Without the uniform percolation assumption, similar results are proved in a companion paper \cite{ZL231} for graphs satisfying condition (1) of Theorem \ref{t68} without the ``uniformly bounded vertex degree" assumption but with the ``uniformly bounded face degree for finite faces" assumption. Here the ``uniformly bounded vertex degree" condition comes from using uniform percolation to obtain ``stability of infinite clusters" as in \cite{RS98}; while the ``uniformly bounded face degree" condition in \cite{ZL231} comes from proving the exponential decay of connectivity, for which we construct a number of disjoint embedded trees separating two vertices; the number of these trees goes to infinity linearly as the distance of two vertices goes to infinity.

We also obtain a universal interval of $p$ for which ``uniform percolation" holds.

\begin{theorem}\label{l613}Let $G=(V,E)$ be a connected, locally finite, non-amenable graph with vertex iso-perimetric constant given by
\begin{align*}
\mathbf{i}_V(G):=\inf\left\{\frac{|\partial_V K|}{|K|}:K\subset V\ \mathrm{finite}\right\}>0
\end{align*}
where $\partial_V K$ consists of all the vertices adjacent to a vertex in $K$ but not in $K$ itself.
Then in the i.i.d.~Bernoulli site percolation on $G$, there is uniform percolation at each $p\in (\frac{1}{\mathbf{i}_V(G)+1},1)$.
\end{theorem}

Moreover, we have
\begin{theorem}\label{t613}Let $G=(V,E)$ be a one-ended, planar graph of uniformly bounded vertex degree and no infinite faces. Assume that one of the following two conditions holds:
\begin{enumerate}
    \item the minimal vertex degree is at least 7; or
    \item the minimal vertex degree is at least 5; and the minimal face degree is at least 4.
\end{enumerate}
Then
\begin{itemize}
\item For all 
$p\in \left(\frac{1}{1+\mathbf{i}_{V}(G)},\frac{\mathbf{i}_V(G)}{1+\mathbf{i}_{V}(G)}\right)$, a.s. there are infinitely many infinite 1-clusters and infinitely many infinite 0-clusters.
\item If $G$ is semi-transitive, then for all
$p\in \left(p_c^{site}(G),1-p_c^{site}(G)\right)$, a.s. there are infinitely many infinite 1-clusters and infinitely many infinite 0-clusters.
\end{itemize}
\end{theorem}

Theorem \ref{t613} seems to be a direct corollary of Theorems \ref{t68}, \ref{l613} and the fact that semi-transitive graphs satisfy the ``uniform percolation" assumption; however, note that Theorem \ref{t613} does not assume that the graph is properly embedded into the plane. Indeed, one can prove that one-ended, planar graph of uniformly bounded vertex degree and no infinite faces can always be properly embedded into the plane; see, e.g. Lemma 4.1 of \cite{Nach}.

Matching graphs were introduced by Sykes and Essam \cite{SE64}
and explored further by Kesten \cite{K82}. Matching pair of graphs are site-percolation analogs of planar duality in bond-percolation, and play an essential role in the proof the main result of the paper.

\begin{definition}\label{df64}Let $G=(V,E)$ be an infinite, connected, locally finite, simple, planar graph.   Fix an embedding to $G$ into the plane. The \textbf{matching graph} $G_*=(V,E_*)$ is the graph whose vertex set is the same as that of $G$; and for $u,v\in V$ and $u\neq v$, $(u,v)\in E_*$ if and only if $u$ and $v$ share a finite face in $G$.
\end{definition}

\begin{definition}The number of ends of a connected graph is the supreme over its finite subgraphs of the number of infinite components that remain after removing the subgraph.
\end{definition}

The graphs considered in this paper may have more than one end; in other words, we allow infinite faces in our graphs. In Definition \ref{df64}, $E_*\setminus E$ contains only edges joining two vertices sharing finite faces; so that $G_*$ is locally finite. If the proper embedding of graph $G$ into $\RR^2$ has only finite faces, then Definition \ref{df64} coincides with the usual definition of matching graphs. It is known from Lemma 4.1 of \cite{Nach} that any one-ended, locally finite planar graph in which every face is finite can be properly embedded into the plane.

\begin{theorem}\label{mt1}
Let $G$ be a one-ended, semi-transitive graph properly embedded in $\RR^2$  with minimal vertex degree at least 7. For i.i.d.~Bernoulli site percolation on $G$ we have
\begin{align*}
p_c^{site}(G_*)+p_u^{site}(G)=1.
\end{align*}
\end{theorem}

Let $n\geq 1$. In \cite{BSte}, it is proved that every graph with a vertex isoperimetric constant at least $n$ contains a tree with vertex isoperimetric constant constant exactly $n$. With the help of this result, we obtain the following theorem

\begin{theorem}\label{l62}
\begin{enumerate}
\item Let $n\geq 2$ be a positive integer. Let $G$ be locally finite planar graph with vertex isoperimetric constant $\mathbf{i}_V(G)\geq n$.    Then for any $p\in\left(\frac{1}{n+1},\frac{n}{n+1}\right)$ a.s. there are infinitely many  infinite open clusters in the i.i.d.~Bernoulli$\left(p\right)$ site percolation on $G$.
\item Let $G$ be a planar graph with vertex isoperimetric constant $\mathbf{i}_V(G)> 1$, uniformly bounded vertex degree and uniformly bounded face degree.   Then  a.s. there are infinitely many  infinite open clusters in the i.i.d.~Bernoulli$\left(\frac{1}{2}\right)$ site percolation on $G$.
\end{enumerate}
\end{theorem}

Instead of a minimal vertex degree assumption, Theorem \ref{l62} assumes a lower bound on the isoperimetric constant and obtain infinitely many infinite open clusters for certain values of $p$. The proof depends on the constructions of embedded trees  on these non-amenable graphs in \cite{BSte} using flows on graphs. The construction applies to non-planar graphs as well; yet it is not immediately clear how to use this construction of embedded trees to prove exponential decays of connectivity, compared to the explicit construction in Lemma \ref{l59}; see \cite{ZL231}.

The organization of the paper is as follows: in Section \ref{sect:up}, we introduce uniform percolation and prove Theorem \ref{t68}. In Section \ref{sect:sc}, we discuss sufficient conditions for uniform percolation and prove Theorem \ref{l613}. In Section \ref{sect:stg}, we discuss semi-transitive graphs and prove Theorem \ref{t613}. In Section \ref{sect:vic} and the appendix, we prove Theorem \ref{l62}.

\subsection{Background and Notation}

All graphs in this paper are assumed simple in the sense that
\begin{itemize}
\item each edge has two distinct endpoints; and
\item there exists at most one edge joining two distinct vertices.
\end{itemize}

Let $G=(V,E)$ be a graph. Once $G$ is suitably embedded in the space $\RR^2$, 
one defines a \emph{face} of $G$ to be a maximal connected subset of
$\RR^2\setminus G$.  Note that faces are open sets, and may be either bounded or unbounded.
 While it may be helpful to think of
a face as being bounded by a cycle of G, the reality can be more complicated in that
faces are not invariably simply connected (if $G$ is disconnected) and their boundaries
are not generally self-avoiding cycles or paths (if $G$ is not 2-connected).

A walk of $G$ is an alternating, finite or infinite  sequence $(\dots,v_0,e_0,v_1,e_1,\dots)$ with 
$e_i=\langle v_i,v_{i+1}\rangle$ for all $i$. Since our graphs are assumed simple, we may refer
to a walk by its vertex-sequence only. A walk is \emph{self-avoiding} if it visits no vertex twice or more, and 
a self-avoiding walk is called a \emph{path}.
A \emph{cycle} $C=(v_0,v_1,\dots,v_n,v_0)$ is a path $(v_0,v_1,\dots,v_n)$ 
such that $v_0\sim v_n$ together with the edge $\langle v_0,v_n\rangle$. 

For $\omega\in\Omega$, we write $u \lra v$ if there exists an open path of $G$ with endpoints $u$ and $v$,
and $x\xleftrightarrow{A} v$ if such a path exists within the set $A \subseteq V$. A similar notation
is used for the existence of infinite open paths. When such open paths exist in the matching graph $G_*$, we
use the relation $\xleftrightarrow{\ast}$.

The boundary of a face is a closed walk; the \emph{degree} of a face $f$, denoted by $|f|$, is the total number of steps of the closed walk representing its boundary. The degree $|f|$ of a face is bounded below by the number of distinct vertices on its boundary and by the number of distinct edges on its boundary; since a vertex or an edge on the boundary may be visited multiple times by the walk. Although in general, the boundary of a face can be more complicated than a cycle, as we shall see in Lemma \ref{lc29}, for graphs under the assumptions of the paper (especially the negative curvature assumption of each vertex), the boundary of each face is a cycle; and the degree of each face is always the number of vertices on its boundary; or the number of edges on its boundary.

\begin{lemma}\label{lc29}(\cite{ZL231})Let $G=(V,E)$ be an infinite, connected, planar graph, properly embedded into $\RR^2$ such that one of the following conditions holds
\begin{enumerate}
\item the minimal vertex degree is at least 7. 
\item the minimal vertex degree is at least 5 and the minimal face degree is at least 4.
\end{enumerate}
Then the boundary of every finite face is a cycle.
\end{lemma}

\begin{lemma}\label{l59}(\cite{ZL231})Let $G=(V,E)$ be an infinite, connected, planar graph, properly embedded into $\RR^2$ such that one of the following 2 conditions holds:
\begin{enumerate}
    \item the minimal vertex degree is at least 7; or
    \item the minimal vertex degree is at least 5; and the minimal face degree is at least 4.
\end{enumerate}
Then there exists a tree $T=(V_{T},E_{T})$ embedded into $G$ such that
\begin{itemize}
    \item the root vertex of $T$ has degree 2; all the other vertices of $T$ has degree 3 or 4;
    \item $V_T\subset V$ and $E_{T}\subset{E}$.
\end{itemize}
\end{lemma}

See Figure \ref{fig:f21} for an example of a tree embedded into the order 7 triangular tilings of the hyperbolic plane (i.e., a vertex transitive graph $G$ drawn in $\HH^2$ such that each vertex has degree 7 and each face has degree 3) satisfying the conditions of Lemma \ref{l59}. See Appendix \ref{sect:B} for explicit constructions of these trees.

\begin{figure}
\centering
\begin{tikzpicture}
\draw[gray, thick] (0,0) -- (1,0);
\draw[gray, thick] (0.6235,-0.7818)--(0,0) -- (0.6235,0.7818);
\draw[red, thick] (0.6235,-0.7818)--(0,0);
\draw[gray, thick] (1,0) -- (0.6235,0.7818)--(-0.2225,0.9749)--(-0.9010,0.4339)--(-0.9010,-0.4339)--(-0.2225,-0.9749)--(0.6235,-0.7818)--(1,0);
\draw[red, thick]  (0.6235,0.7818)--(-0.2225,0.9749);

\draw[red, thick] (-0.9010,0.4339)--(-0.9010,-0.4339);
\draw[red, thick] (-0.2225,-0.9749)--(0,0);
\draw[red, thick]  (0,0)--(-0.2225,0.9749);
\draw[gray, thick] (-0.9010,-0.4339)--(0,0) -- (-0.9010,0.4339);
\draw[gray, thick] (1,0) -- (1.1940,0.5972);
\draw[red, thick] (1.1940,0.5972)
--(0.6235,0.7818);
\draw[gray, thick] (1,0) -- (1.5657,0.2724)--(1.1940,0.5972);
\draw[gray, thick] (1,0) -- (1.5879,-0.2206)--(1.5657,0.2724);
\draw[gray, thick] (1,0) --(1.2468,-0.5774) --(1.5879,-0.2206);
\draw[gray, thick] (0.6235,-0.7818)--(1.2468,-0.5774)--(1.2020,-1.0911);
\draw[red,thick]
(1.2020,-1.0911)--(0.6235,-0.7818);
\draw[gray, thick] (1.2020,-1.0911)--(0.7997,-1.4136);
\draw[red, thick]
(0.7997,-1.4136)--(0.6235,-0.7818);
\draw[gray, thick] (0.7997,-1.4136)--(0.2886,-1.3458);
\draw[red, thick]
(0.2886,-1.3458)--(0.6235,-0.7818);
 \draw[gray, thick] (0.2886,-1.3458)
--(-0.2225,-0.9749);
\draw[gray, thick] (0.2886,-1.3458)--(-0.1374,-1.6006);
\draw[red,thick]
(-0.1374,-1.6006)--(-0.2225,-0.9749);
\draw[gray, thick] (-0.1374,-1.6006)--
(-0.6160,-1.4688);
\draw[red, thick]
(-0.6160,-1.4688)--(-0.2225,-0.9749);
\draw[gray, thick] (-0.6160,-1.4688)--(-0.8514,-1.0318)--(-0.2225,-0.9749);
\draw[red, thick] (-0.9010,-0.4339)--(-0.8514,-1.0318);
\draw[red, thick]
(-0.9010,-0.4339)--(-1.2945,-0.8848);
\draw[gray, thick]
(-1.2945,-0.8848)--(-0.8514,-1.0318);
\draw[gray, thick] (-1.2945,-0.8848)--(-1.4945,-0.4629);
\draw[red,thick]
(-1.4945,-0.4629)--(-0.9010,-0.4339);
\draw[gray, thick] (-1.4945,-0.4629)--(-1.3278,-0.0268)--(-0.9010,-0.4339);
\draw[red, thick] (-1.3278,-0.0268)--(-0.9010,0.4339);
\draw[gray, thick] (-1.3278,-0.0268)--(-1.5289,0.4240);
\draw[gray,thick]
(-1.5289,0.4240)--(-0.9010,0.4339);
\draw[gray, thick] (-1.5289,0.4240)--(-1.3421,0.8809)--(-0.9010,0.4339);
\draw[gray, thick] (-1.3421,0.8809)--(-0.8827,1.0616);
\draw[red,thick]
(-0.8827,1.0616)--(-0.9010,0.4339);
\draw[red, thick] (-0.8827,1.0616)--(-0.2225,0.9749);
\draw[gray, thick] (-0.8827,1.0616)--(-0.6161,1.5120)--(-0.2225,0.9749);
\draw[gray, thick] (-0.6161,1.5120)--(-0.1062,1.6305)--(-0.2225,0.9749);
\draw[gray, thick] (-0.1062,1.6305)--(0.3317,1.3440);
\draw[red, thick]
(0.3317,1.3440)--(-0.2225,0.9749);
\draw[gray, thick] (0.3317,1.3440)--(0.6235,0.7818);
\draw[red, thick] (0.3317,1.3440)--(0.8282,1.3812);
\draw[gray, thick]
(0.8282,1.3812)--(0.6235,0.7818);
\draw[gray, thick] (0.8282,1.3812)--(1.1982,1.0480);
\draw[red,thick]
(1.1982,1.0480)--(0.6235,0.7818);
\draw[gray, thick] (1.1982,1.0480)--(1.1940,0.5972);
\draw[gray, thick] (-1.5286,0.4240)--(-1.7498,0.2292);
\draw[gray, thick]
(-1.7498,0.2292)--(-1.3278,-0.0268);
\draw[gray, thick] (-1.7498,0.2292)--(-1.8205,-0.0567);
\draw[red,thick]
(-1.8205,-0.0567)--(-1.3278,-0.0268);
\draw[gray, thick] (-1.8205,-0.0567)--(-1.7158,-0.3320);
\draw[red,thick]
(-1.7158,-0.3320)--(-1.3278,-0.0268);
\draw[gray, thick] (-1.7158,-0.3320)--(-1.4945,-0.4629);
\draw[gray, thick] (0.8282,1.3812)--(0.7186,1.6573);
\draw[red, thick]
(0.7186,1.6573)--(0.3317,1.3440);
\draw[gray, thick] (0.7186,1.6573)--(0.4713,1.8219);
\draw[red, thick]
(0.4713,1.8219)--(0.3317,1.3440);
\draw[gray, thick] (0.4713,1.8219)--(0.1743,1.8164)--(0.3317,1.3440);
\draw[gray, thick] (0.1743,1.8164)--(-0.1062,1.6305);
\end{tikzpicture}
\caption{Tree embedding in a order-7 triangular tiling of the hyperbolic plane: the order-7 triangular tiling is represented by black lines, and the embedded tree is represented by red lines.}\label{fig:f21}
\end{figure}
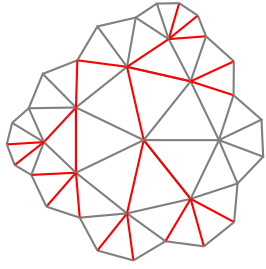

\begin{lemma}\label{la68}(\cite{ZL231})Let $G=(V,E)$ be an infinite, connected, locally finite graph. Let $\mathcal{A}_1$ be the event that there exists a unique infinite 1-cluster. Then for any $u,v\in V$
\begin{align*}
\mathbb{P}_p(u\leftrightarrow v)\geq \mathbb{P}_p(u\leftrightarrow\infty)\mathbb{P}_p(v\leftrightarrow\infty)\mathbb{P}_p(\mathcal{A}_1)
\end{align*}
\end{lemma}

\begin{lemma}\label{l96}(\cite{ZL231})Let $G=(V,E)$ be an infinite, connected graph properly embedded into $\RR^2$ with minimal vertex degree at least 7 . Then for each $p< 1-p_c^{site}(T)$, there exists $c_p>0$, such that for any $u,v\in V$, 
\begin{align}
  \mathbb{P}_p(u\leftrightarrow v) \leq \mathbb{P}_p(u\xleftrightarrow{*} v)e^{-c_pd_{G_*}(u,v)}.\label{cl1}
\end{align}
where $u\leftrightarrow v$ means that $u$ and $v$ are in the same 1-cluster, while $u\xleftrightarrow{*} v$ means that $u$ and $v$ are in the same 1-*-cluster (1-cluster in the graph $G_*$). Moreover, 
\begin{align}
\mathbb{P}_p(\partial_V^*u\leftrightarrow \partial_V^*v) \leq e^{-c_pd_{G_*}(u,v)}.\label{cl2}
\end{align}
where $\partial_V^* u$ ($\partial_V^* v$) consists of all the vertices in $V$ adjacent to $u$ (resp.\ $v$) in $G_*$.
\end{lemma}

\begin{lemma}\cite{ZL231}\label{lt1}
Let $G$ be an infinite, connected graph properly embedded in $\RR^2$ such that the minimal vertex degree at least 7. Then, for any $p\in  (p_c^{site}(G),1-p_c^{site}(G_*))$, a.s.~there are infinitely many infinite 1-clusters in the i.i.d.~Bernoulli($p$) site percolation on $G$. Moreover, it is the case that $p_c^{site}(G)<\frac{1}{2}$ , so the above interval is
invariably non-empty.
\end{lemma}

\section{Uniform Percolation}\label{sect:up}

In this section, we discuss the existence of infinitely many infinite clusters when $p\in (p_c^{site},1-p_c^{site})$ under the further assumption of ``uniform percolation".

\begin{definition}\label{df61}(\cite{RS98})Let $p\in[0,1]$. On a graph $G=(V,E)$ there is uniform percolation at level $p$ if
\begin{align*}
\lim_{N\rightarrow\infty}\inf_{x\in V}\mathbb{P}(B(x,N)\ \mathrm{intersects\ an\ infinite\ cluster\ in\ i.i.d.~Bernoulli(p)\ site\ percolation})=1
\end{align*}
where $B(x,N)$ is the ball in $G$ centered at $x$ with radius $N$.
\end{definition}

\begin{proposition}\label{p62}(Theorem 3.1 in \cite{RS98})Let $G=
(V,E)$ be an infinite connected graph of uniformly bounded degree.
Let ${U(v)}_{v\in V}$ be i.i.d. random variables, uniformly distributed in $[0, 1]$, so
their joint distribution is a product measure on $[0, 1]^V$. For $p\in [0, 1]$,
denote $V_p:=\{v\in V : U(v)\leq p\}$ and let $G_p = (V_p , E_p )$, where $E_p\subseteq E$ consists of all the edges joining two vertices in $V_p$. Assume that $0 < p_1 < p_2\leq 1$, and that there is uniform percolation at level $p_1$. 
Then almost surely, 
any infinite cluster of $G_{p_2}$ contains at least one infinite cluster of $G_{p_1}$.
\end{proposition}

For $0<p_1<p_2<1$, we call the coupling between i.i.d.~ Bernoulli($p_1$) percolation and i.i.d.~ Bernoulli($p_2$) percolation by the same collection of i.i.d.~uniform random variables on vertices as described in Proposition \ref{p62} standard coupling.

Note that the conclusion of Proposition \ref{p62} with stronger assumptions that $G$ is quasi-transitive and unimodular was proved in \cite{HP}.

\begin{lemma}\label{l63}Let $G=(V,E)$ be an infinite, connected, planar graph of uniformly bounded degree. Assume that $G$ can be properly embedded into $\RR^2$ such that one of the following 2 conditions holds:
\begin{enumerate}
    \item the minimal vertex degree is at least 7; or
    \item the minimal vertex degree is at least 5; and the minimal face degree is at least 4.
\end{enumerate} 
Assume that for every $p\in \left(p_0,\frac{1}{2}\right)$ there is uniform percolation, where $p_0\in \left[p_c^{site},\frac{1}{2}\right)$. Then 
\begin{enumerate}[label=(\alph*)]
\item for all $p\in \left(p_0,\frac{1}{2}\right)$, a.s. there are infinitely many infinite 1-clusters and infinite 0-clusters have infinitely many ends.
\item for all $p\in \left(\frac{1}{2},1-p_0\right)$, a.s. there are infinitely many infinite 0-clusters and infinite 1-clusters have infinitely many ends.
\end{enumerate}
\end{lemma}

\begin{proof}It suffices to prove part (a); part (b) follows from part (a) by symmetry.

It follows directly from Proposition \ref{p62} and Theorem \ref{l59} (a) that for all $p\in \left(p_0,\frac{1}{2}\right)$, a.s. there are infinitely many infinite 1-clusters.

Recall that in the proof of Lemma \ref{l59}, we constructed a tree $T$ which is a subgraph of $G$, and $p_c^{site}(T)<\frac{1}{2}$. For $q\in\left(p_0,\frac{1}{2}\right]$, let $\omega_q\in[0,1]^V$ be a i.i.d.~Bernoulli($q$) site percolation on $G$, and always consider the coupling
given in Proposition \ref{p62} for different $q$'s.
Let $s$ be an arbitrary finite-length binary number,
such that $v_{s}$ has two children in the tree $T$.
Assume that
\begin{itemize}
    \item $\omega_{\frac{1}{2}}(v_s)=0$, $\omega_{\frac{1}{2}}(v_{s0})=0$, $\omega_{\frac{1}{2}}(v_{s00})=0$, $\omega_{\frac{1}{2}}(v_{s1})=0$,
    $\omega_{\frac{1}{2}}(v_{s11})=0$,
    the subtree with root $v_{s00}$ has an infinite 0-cluster including the vertex $v_{s00}$, and the subtree with root $v_{s11}$ has an infinite 0-cluster including the vertex $v_{s11}$. Let $\xi_0$ denote this infinite 0-cluster passing through $v_s$, $v_{s0}$, $v_{s00}$, $v_{s1}$, $v_{s11}$.
    \end{itemize}
By the Borel-Contelli lemma, a.s. the above event occurs for infinitely many $s$, each of which has
two children in $T$ and the subtree rooted with distinct $s$ are disjoint. Let $\eta_1$ be an infinite 1-cluster in the subtree of $T$ rooted vertex $v_{s01}$ in $\omega_{\frac{1}{2}}$, and let $\rho_1\supseteq \eta_1$ be the infinite 1-cluster in $G$ containing $\eta_1$. Then by Proposition \ref{p62} in $\omega_{q}$ of $G$, a.s.~an infinite 1-cluster $\tilde{\rho}_1$ is contained in $\rho_1$, and a.s.~an infinite 0-cluster $\tilde{\xi}_0$ contains $\xi_0$. Note that $\tilde{\xi}_0$ extends to infinity on both the left side and the right side of $\tilde{\rho}_1$, and $\tilde{\rho}_1$ extends to infinity as well. Since this occurs infinitely often with disjoint $\tilde{\rho}_1$'s, we conclude that for all $q\in \left(p_0,\frac{1}{2}\right)$, a.s.~infinite 0-clusters have infinitely many ends.
\end{proof}

\begin{lemma}\label{l65}Let $G=(V,E)$ be an infinite, connected, planar graph of uniformly bounded vertex degree. Assume that $G$ can be properly embedded into $\RR^2$ such that one of the following 2 conditions holds:
\begin{enumerate}
    \item the minimal vertex degree is at least 7; or
    \item the minimal vertex degree is at least 5; and the minimal face degree is at least 4.
\end{enumerate} 
Let $p_0\in \left[p_c^{site}(G_*),\frac{1}{2}\right)$. Let $G_1$ be a graph satisfying $G\subset G_1\subset G_*$. If for every $p\in \left(p_0,\frac{1}{2}\right)$ there is uniform percolation in $G_1$ and $G_1$ has uniformly bounded vertex degree; then 
for every $p\in \left(\frac{1}{2},1- p_0\right)$ in the i.i.d.~Bernoulli($p$) site percolation on $G$, one can find a sequence of events $\{J_i\}_{i\geq 1}$; satisfying all the following conditions
\begin{itemize}
    \item $\{J_i\}_{i\geq 1}$ are mutually independent; and
    \item there exists $c_0>0$ independent $i$ such that $\mathbb{P}(J_i)\geq c_0$ for all $i$; and
    \item there exist a sequence of vertices $\{x_i\}_{i\geq 1}$ and a positive integer $1\leq N_1<\infty$ independent of $i$ such that $B(x_i,N_1)$ are pairwise disjoint; and if $J_i$ occurs, then $[G\setminus B(x_i,N_1)]$ has at least 3 infinite 1-clusters in the i.i.d.~Bernoulli($p$) site percolation.
\end{itemize}
\end{lemma}

\begin{proof}Let $p\in \left(\frac{1}{2},1-p_0\right)$; then there is uniform percolation at level $1-p$ on $G_1$. We shall use 1-$G_1$-cluster (resp.~0-$G_1$-cluster) to denote 1-cluster (resp.~0-cluster) in $G_1$.

Let $N_0$ be such that
\begin{align*}
&\inf_{x\in V}\mathbb{P}(B(x,N_0)\ \mathrm{intersects\ an\ infinite\ 1}-G_1-\mathrm{cluster\ in\ i.i.d.~Bernoulli}(1-p)\ \mathrm{site\ percolation})\\
&\geq 1-\epsilon.
\end{align*}
where $\epsilon>0$ is a small positive number to be determined later.

Recall from the proof of Lemma \ref{l59} that there exists a tree $T$ which is a subgraph of $G$. Let $r,s,t$ be 3 finite-length sequences consisting of $0,\frac{1}{2},1$ obtained as follows.
\begin{itemize}
    \item there exists a finite-length sequence $u$ (to be determined later) consisting of $0,\frac{1}{2},1$ such that 
    \begin{align}
        r=u0;\ s=u10;\ t=u110.\label{ru}
    \end{align}
\end{itemize}

It is straightforward to check that $r,s,t$ can be chosen to satisfy
all the following conditions:
\begin{itemize}
    \item $v_r$, $v_s$,$v_t$ are 3 vertices of $T$; and
    \item the subtrees $T_r,T_s,T_t$ of $T$ rooted at $v_r$, $v_s$, $v_t$ are isomorphic and disjoint; and
    \item each one of $v_r$, $v_s$, $v_t$ has two children in $T$.
\end{itemize}

Let $\omega\in\{0,1\}^V$ be the i.i.d.~Bernoulli($p$) site percolation configuration on $G$. Let $\tilde{\omega}\in\{0,1\}^V$ be the i.i.d.~Bernoulli($\frac{1}{2}$) site percolation configuration on $G$,
such that $(\tilde{\omega},\omega)$ form a standard coupling. More precisely, assign an i.i.d.~uniform random variable $U_z$ in $(0,1)$ to each vertex $z\in V$; then define
\begin{itemize}
\item $\omega(z)=\tilde{\omega}(z)=1$ if $U_z\in \left(0,\frac{1}{2}\right]$; and
\item $\omega(z)=1;\tilde{\omega}(z)=0$ if $U_z\in \left(\frac{1}{2},p\right]$; and
\item $\omega(z)=\tilde{\omega}(z)=0$ if $U_z\in \left(p,1\right]$.
\end{itemize}
Then by Proposition \ref{p62}, each infinite 0-*-cluster in $\tilde{\omega}$ contains at least one infinite 0-*-cluster in $\omega$.

For each $j\in\{r,s,t\}$, note that a.s.~in the subtree $T_{v_{j01}}$ of $T$ rooted at $v_{j01}$, there is an infinite 0-cluster  in $\tilde{\omega}$, and therefore an infinite 0-*-cluster in $\tilde{\omega}$. Let $l_{j,1}$ (resp.\ $l_{j,2}$) be the left (resp.\ right) boundary of $T_{v_{j01}}$. Let $G_j$ be the subgraph of $G$ bounded by $l_{j,1}$ and $l_{j,2}$ containing the tree $T_{v_{j01}}$. Let $y_j$ be a vertex in $G_j$, such that $B(y_j,N_0+1)\subset G_j$. Then  the trees $T_{v_{j00}}$ and $T_{v_{j11}}$ are disjoint from $B(y_j,N_0+1)$.
 Let $E_j$ be the event defined by
\begin{itemize}
    \item $\tilde{\omega}(v_j)=1$, $\tilde{\omega}(v_{j0^i})=\tilde{\omega}(v_{j1^i})=1$, for all $1\leq i\leq 2$. The subtree of $T$ rooted at $v_{j00}$ (resp.\ $v_{j11}$) has an infinite 1-cluster including the vertex $v_{j00}$ (resp.\ $v_{j11}$) in $\tilde{\omega}$, and hence in $\omega$.
    \end{itemize}
 
 Then with probability at least $(1-\epsilon)$, $B(y_j,N_0)$ intersects an infinite 0-$G_1$-cluster of $G$ in $\omega$.

Then there exists a constant $c_1>0$ (independent of $N_0$), such that 
\begin{align*}
    \mathbb{P}(E_j)\geq c_1.
\end{align*}
Let $I:=\{r,s,t\}$.
Note that $\{E_j\}_{j\in I}$ are independent; 
we have
\begin{align*}
    \mathbb{P}(\cap_{j\in I}E_j)\geq (c_1)^3
\end{align*}
Choose $\epsilon=\frac{(c_1)^3}{2}$. For $k\in I$, Let $F_j$ be the event that $B(y_j,N_0)$ intersects an infinite 0-$G_1$-cluster. 

Note that $\{E_j\cap F_j\}_{j\in I}$ are mutually independent, since each $E_j\cap F_j$ depends only on the configurations within the subgraph of $G$ bounded by the left and right boundaries of the tree $T_j$.
Moreover,
\begin{align*}
    \mathbb{P}([\cap_{j\in I}E_j]\cap F_k)\geq P(F_k)-P([\cap_{j\in I}E_j]^c)\geq 1-\epsilon+(c_1)^3-1=\frac{(c_1)^3}{2}
\end{align*}
Hence we have
\begin{align*}
    \mathbb{P}(F_k|[\cap_{j\in I}E_j])\geq 
    \frac{(c_1)^3}{2}
\end{align*}
By the independence of $F_k$'s conditional on $\cap_{j\in I}E_j$, we obtain 
\begin{align*}
    \mathbb{P}(\cap_{k\in I}F_k|[\cap_{j\in I}E_j])\geq 
    \left(\frac{(c_1)^3}{2}\right)^3
\end{align*}
and therefore
\begin{align*}
    \mathbb{P}(\cap_{j\in I}[E_j\cap F_j])=\mathbb{P}(\cap_{k\in I}F_k|[\cap_{j\in I}E_j])\mathbb{P}([\cap_{j\in I}E_j])\geq \left(\frac{(c_1)^3}{2}\right)^3(c_1)^3.
\end{align*}
Recall that $u$ and $r,s,t$ satisfy (\ref{ru}). Let $N_1>0$ be such that 
\begin{align*}
B(u,N_1)\supseteq \cup_{i}B(y_i,N_0).
\end{align*}
Let 
\begin{align*}
    J_1:=\cap_{j\in I}[E_j\cap F_j]
\end{align*}
Let $x_1=u$. It is straightforward to check that if $J_1$ occurs, $[G\setminus B(v,N_1)]$ has at least 3 infinite 1-clusters in i.i.d.~Bernoulli($p$) site percolation.
See Figure \ref{fig:tt5}.

\begin{figure}
    \centering
    \includegraphics[width=1.06\textwidth]{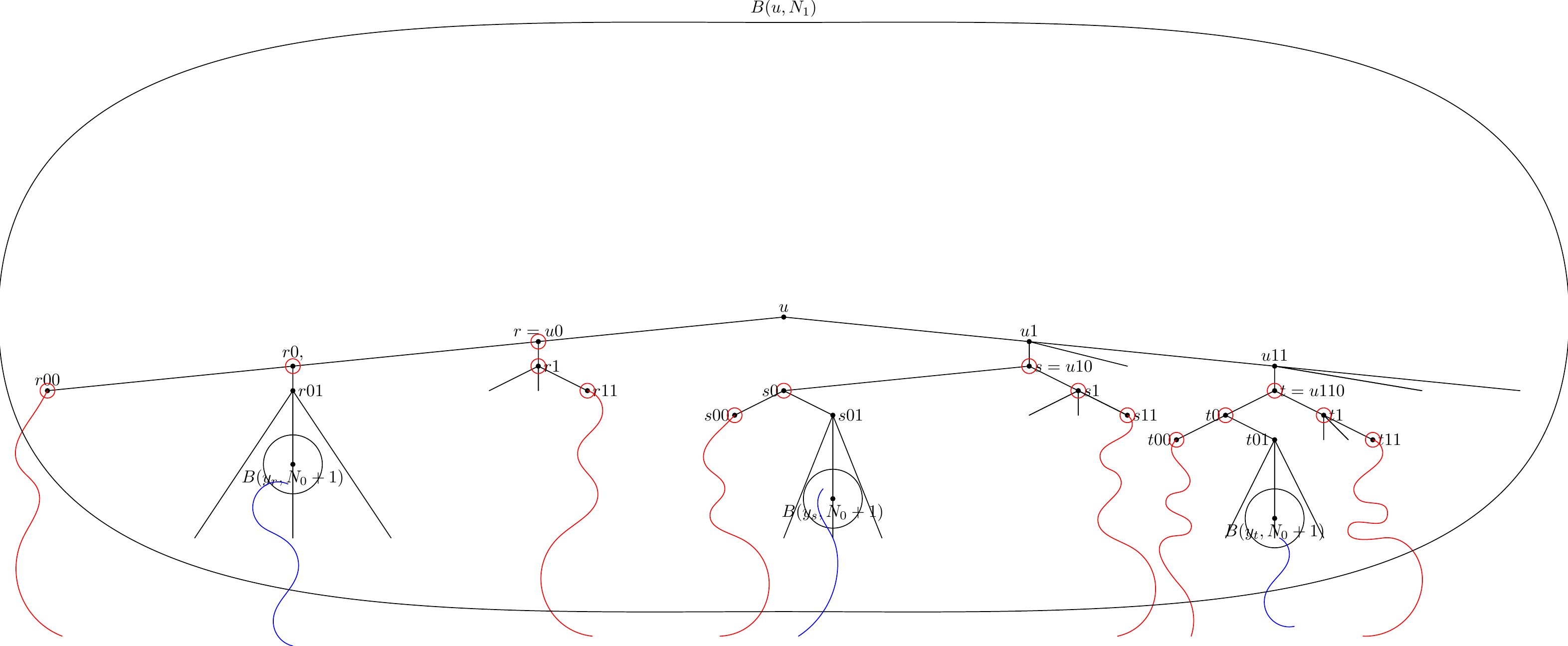}
    \caption{A site configuration in which after removing a finite graph $B(u,N_1)$, at least 3 infinite 1-clusters are remained. State ``1" is represented by red, and state ``0" is represented by blue. The finite subgraph $B(u,N_1)$ is bounded by the biggest oval. }.
    \label{fig:tt5}
\end{figure}
    
Choose $c_0=\left(\frac{(c_1)^3}{2}\right)^3(c_1)^3$. Note that the root $v$ of $T$ can be chosen to be any vertex of $G$. We may choose $u$ such that $B(u,N_1)$ is contained in the region bounded by the left and right boundary of a subtree of $T$ with root $v_{w}$ (we shall also write it as a subtree of $T$ rooted at $w$), where $w$ is a finite length sequence consisting of $0$,
$\frac{1}{2}$, $1$ and $w=f0$. Then consider the subtrees of $T$ rooted at $w10$, $w110$, $w1110$,\ldots; note that they are disjoint. Choose $x_i$ such that $B(x_i,N_1)$ is in the region bounded by the left and right boundary of the subtree of $T$ rooted at $w1^{k-1}0$. With $u$ replaced by $x_i$ and repeat the process above, we can find a sequence of events $\{J_i\}_{i\geq 1}$, such that different $J_i$'s depend only on vertices in disjoint graphs bounded by boundaries of disjoint subtrees. Then the lemma follows.
\end{proof}

\begin{lemma}\label{l66}Let $G=(V,E)$ be an infinite, connected, planar graph of uniformly bounded vertex degree. Assume that $G$ can be properly embedded into $\RR^2$ such that one of the following 2 conditions holds:
\begin{enumerate}
    \item the minimal vertex degree is at least 7; or
    \item the minimal vertex degree is at least 5; and the minimal face degree is at least 4.
\end{enumerate} 
Let $p_0\in \left[p_c^{site}(G_*),\frac{1}{2}\right)$. Let $G\subseteq G_1\subseteq G_*$. If for every $q\in \left(p_0,\frac{1}{2}\right)$ there is uniform percolation in $G_1$ and $G_1$ has uniformly bounded vertex degree, then there are infinitely many infinite 0-clusters in the i.i.d.~Bernoulli($q$) site percolation on $G$.
\end{lemma}

\begin{proof}Let $p=1-q$. 
As in Lemma \ref{l65}, choose a sequence $\{x_n\}_{n\geq 1}$ of vertices in $G$ satisfying all the following conditions:
\begin{itemize}
\item $B(x_i,N_2)$ are pairwise disjoint; and
\item Let $H_i$ be the event that $G\setminus B(x_i,N_2)$ has 3 infinite clusters in i.i.d.~Bernoulli($p$) site percolation; then there exist mutually independent events $\{J_i\}$ satisfying $J_i\subseteq H_i$, and $\mathbb{P}(J_i)\geq c_1$, where $c_1>0$ is a constant independent of $i$. 
\end{itemize}
Once we find the sequence $\{x_i\}_{i\geq 1}$ satisfying all the above conditions, let $\tilde{J}_i$ be obtained from $J_i$ by making all the sites in $B(x_i,N)$ to be closed; we have
\begin{align*}
\mathbb{P}(\tilde{J}_i)\geq c_1\left(\frac{1-p}{p}\right)^{L(N_2)}>0
\end{align*}
where for each positive integer $1\leq N<\infty$
\begin{eqnarray*}
L(N):=\sup\{x\in V:|B(x_i,N)|\}<\infty.
\end{eqnarray*}
Here $L(N)<\infty$ follows from the fact that the graph $G$ has uniformly bounded vertex degree. Since $B(x_i,N_2)$ are pairwise disjoint, and $J_i$'s are independent; we infer that $\tilde{J}_i$'s are independent. The sum of probabilities of all $\tilde{J}_i$'s are infinite, then the Borel-Contelli lemma implies that $\tilde{J}_i$ occurs infinitely often a.s.; as a result we obtain infinitely many infinite 1-clusters a.s.
\end{proof}


\noindent{\textbf{Proof of Theorem \ref{t68}}.} Since each infinite 1-cluster is contained in an infinite 1-*-cluster, from definition \ref{df61}, it is straightforward to check that if for every $p\in \left(p_0,\frac{1}{2}\right)$ there is uniform percolation in $G$, then for every $p\in \left(p_0,\frac{1}{2}\right)$ there is uniform percolation in $G_*$. Then the theorem follows from Theorem \ref{l59}, Lemma \ref{l63} and Lemma \ref{l66}.
$\hfill\Box$

\section{Sufficient conditions for uniform percolation.}\label{sect:sc}
In this section we discuss the conditions to guarantee ``uniform percolation".

\begin{proposition}\label{l69}Let $G=(V,E)$ be an infinite, connected  graph  that can be properly embedded into $\RR^2$.
Suppose that one of the following 2 conditions holds:
\begin{enumerate}
    \item the minimal vertex degree is at least 7; or
    \item the minimal vertex degree is at least 5; and the minimal face degree is at least 4; or
\end{enumerate}
Let $T$ be the tree embedded into $G$ as a subgraph as constructed in Proposition \ref{l69}. Then there is uniform percolation in $G$ for every $p\in \left(p_c^{site}(T),\frac{1}{2}\right)$.
\end{proposition}

\begin{proof}Note that the tree $T$ embedded into $G$ is self-similar, hence there is uniform percolation in $T$ for every $p\in \left(p_c^{site}(T),\frac{1}{2}\right)$. Then the lemma follows from the following facts
\begin{enumerate}
\item every ball of radius $N$ in $T$ is contained in a ball of radius $N$ in $G$; and
\item for every $p\in \left(p_c^{site}(T),\frac{1}{2}\right)$, and infinite 1-cluster in $T$ is contained in an infinite 1-cluster in $G$.
\end{enumerate}
\end{proof}

\noindent{\textbf{Proof of Theorem \ref{l613}}.} The proof repeated uses an exploration process constructed in \cite{bs96}, until we exhaust all the vertices in a ball $B(v,M)$.

More precisely, consider the following inductive procedure for constructing the percolation cluster of $v\in V$. If $v$ is closed, set $C_n=\emptyset$ for each $n$. Otherwise set $C_1=\{v\}$ and $W_1=\emptyset$. For each $n\geq 2$, if $\partial_V C_{n-1}\subseteq W_{n-1}$, set $C_m=C_{n-1}$ and $W_m=W_{n-1}$ for all $m\geq n$. Otherwise, choose a vertex $w_n\in \partial_V C_{n-1}\setminus W_{n-1}$. If $w_n$ is open, let $C_n=C_{n-1}\cup \{w_n\}$, and $W_n=W_{n-1}$; otherwise let $C_n=C_{n-1}$, and $W_n=W_{n-1}\cup\{w_n\}$.

Let $C_v=\cup_{n}C_n$., If $C_v$ is nonempty and finite, then there exists a finite $N$, such that $\partial_V C_N=W_N=W_v$.

For each $v\in V$, recall that $B(v,M)$ consists of all the vertices in $G$ whose distance to $v$ is at most $M$. Label all the vertices in $B(v,M)$ by $v_0,v_1,\ldots,v_{|B(v,m)|}$, such that $v=v_0$. Then $B(v,M)$ does not intersect infinite 1-clusters if and only if 
\begin{itemize}
    \item We construct $C_{v_0}$ and $W_{v_0}$, then there exists a finite $N_0$, such that $C_{v_0}\cup W_{v_0}=N_0$ and $\partial_V C_{v_0}=W_{v_0}$
    \item Find the smallest vertex $z_1\in B(v,m)\setminus [C_{v_0}\cup W_{v_0}]$, and construct $C_{z_1}$ and $W_{z_1}$,
    , such that $|C_{z_1}\cup W_{z_1}|<\infty$ and $\partial_V C_{z_1}\subset W_{z_1}\cup W_{v_0}$
    \item repeat the process above until we exhaust all the vertices in $B(v,M)$ and each vertex is in a finite 1-cluster.
    Since $G$ is locally finite (each vertex has finite degree), $B(v,M)$ is finite, there exists a finite positive integer $K$, such that $B(v,M)\subseteq \cup_{k\in [K]}[C_{z_k}\cup W_{z_k}]\cup C_{v_0}\cup W_{v_0}$.
\end{itemize}
Let 
\begin{align*}
    C=C_{v_0}\cup [\cup_{k\in [K]}C_{z)k}];\qquad
    W=W_{v_0}\cup [\cup_{k\in [K]}W_{z)k}].
\end{align*}
Then 
\begin{align*}
    |C\cup W|<\infty;\qquad \partial C=W.
\end{align*}
Let
\begin{align*}
    \infty>L:=|C\cup W|\geq |B(v,M)|\geq M;
\end{align*}
where the last inequality follows from the fact that $G$ is connected.
By the definition of the vertex isoperimetric constant we have
\begin{align*}
    |W|\geq \mathbf{i}_V(G)|C|
\end{align*}
Hence we have 
\begin{align*}
    |W|\geq \frac{\mathbf{i}_V(G)}{\mathbf{i}_V(G)+1}L.
\end{align*}
However $\mathbb{E}|W|=(1-p)L$. By the Hoeffding's inequality, when $p>\frac{1}{\mathbf{i}_V(G)+1}$
\begin{align*}
    \mathbb{P}(|C\cup W|=L; W=\partial C)&\leq \mathbb{P}\left(||W|-E|W||\geq L\left(p-\frac{1}{\mathbf{i}_V(G)+1}\right)\right)\\
    &\leq \exp\left[-\frac{\left(p-\frac{1}{\mathbf{i}_V(G)+1}\right)^2 L}{2}\right]
\end{align*}

Then we have
\begin{align*}
    \mathbb{P}(B(v,M)\nleftrightarrow \infty)&\leq 
\sum_{L\geq |B(v,M)|} \mathbb{P}(|C\cup W|=L; W=\partial C)\\
&\leq \left(1-\exp\left[-\frac{\left(p-\frac{1}{\mathbf{i}_V(G)+1}\right)^2 }{2}\right]\right)^{-1}\exp\left[-\frac{\left(p-\frac{1}{\mathbf{i}_V(G)+1}\right)^2|B(v,M)| }{2}\right]\\
&\leq \left(1-\exp\left[-\frac{\left(p-\frac{1}{\mathbf{i}_V(G)+1}\right)^2 }{2}\right]\right)^{-1}\exp\left[-\frac{\left(p-\frac{1}{\mathbf{i}_V(G)+1}\right)^2M }{2}\right];
\end{align*}
which converges to 0 uniformly in $v$ as $M\rightarrow\infty$. Then the lemma follows from Definition \ref{df61}.
$\hfill\Box$

A refined exploration process in \cite{BNY11} can give larger region of $p$ for uniform site percolation when the graph $G$ is regular, (i.e., each vertex has the same degree) and large girth (i.e. the length of minimal cycle).

\section{Semi-transitive Graphs}\label{sect:stg}

In this section, we discuss semi-transitive graphs and prove theorem \ref{t613}.

\begin{definition}\label{d31}
A graph $G=(V,E)$ is called semi-transitive if there is a finite set $V_F\subset V$ s.t. for any vertex $x\in V$, there is a vertex $y\in V_{F}$ and an injective graph homomorphism of $G$ that maps $y$ to $x$.
\end{definition}

Semi-transitive graphs were defined in Section 2 of \cite{hps98}. It is straightforward to check that semi-transitive graphs include quasi-transitive graphs but not vice versa. 
Note that whether or not percolation occurs is measurable with respect to the tail-$\sigma$-algebra, hence the probability that percolation occurs is either 0 or 1.

\begin{lemma}\label{le68}Let $G=(V,E)$ be an infinite, connected, locally finite, semi-transitive graph. Then for each $p\in (p_c^{site}(G),1)$, there is uniform percolation in the i.i.d.~Bernoulli($p$) site percolation on $G$.
\end{lemma}

\begin{proof}See Section 2 of \cite{hps98}.
\end{proof}

\begin{corollary}\label{c69}Let $G=(V,E)$ be an infinite, connected, planar graph of uniformly bounded vertex degree. Assume that $G$ can be properly embedded into $\RR^2$ such that one of the following 2 conditions holds:
\begin{enumerate}
    \item the minimal vertex degree is at least 7; or
    \item the minimal vertex degree is at least 5; and the minimal face degree is at least 4.
\end{enumerate} 
If $G$ is semi-transitive, then for all $p\in \left(p_c^{site}(G),1-p_c^{site}(G)\right)$, a.s. there are infinitely many infinite 1-clusters and infinitely many infinite 0-clusters.
\end{corollary}
\begin{proof}The corollary follows from Theorem \ref{t68} and Lemma \ref{le68}.
\end{proof}

\begin{example}\label{e610}Let $\overline{G}=(\overline{V},\overline{E})$ be a vertex-transitive triangulation of the hyperbolic plane $\HH^2$ in which each vertex has degree 12 and each face has degree 3. Let $l$ be a directed doubly-infinite self-avoiding path of $\overline{G}$ such that at each vertex $v$ along $l$, both the left hand side of $l$ and right hand side of $l$ have 6 incident faces of $v$. Let $G$ be the graph on the right side of $l$ (including l). Then $G$ is semi-transitive. By \cite{Bab97}, $\overline{G}$ can be embedded in the hyperbolic plane such that each face is a regular triangle; hence $G$ inherits an embedding into the hyperbolic plane such that each ege is a geodesic. Then by corollary \ref{c69}, for all $p\in \left(p_c^{site}(G),1-p_c^{site}(G)\right)$, a.s. there are infinitely many infinite 1-clusters and infinitely many infinite 0-clusters. Moreover $p_c^{site}(G)=p_c^{site}(\overline{G})$ by Corollary 4.4. of \cite{bsjams}.
\end{example}

\begin{example}Let $\overline{G}=(\overline{V},\overline{E})$ and $l$, $G$ be given as in Example \ref{e610}. Let $G_1$ be a graph obtained from $G$ as follows:
\begin{itemize}
    \item for each face without a vertex along $l$, add an extra vertex in the center and join this vertex to every vertex of the face by an edge.
\end{itemize}

Then $G_1$ is semi-transitive and has an embedding into the hyperbolic plane such that each ege is a geodesic. Then by corollary \ref{c69}, for all $p\in \left(p_c^{site}(G_1),1-p_c^{site}(G_1)\right)$, a.s. there are infinitely many infinite 1-clusters and infinitely many infinite 0-clusters in the i.i.d.~Bernoulli($p$) site percolation on $G_1$.
\end{example}

\noindent{\textbf{Proof of Theorem \ref{t613}}.} When the graph $G$ satisfies the assumption of Theorem \ref{t613}, we place a vertex at the center of each face of $G$ and connect this new vertex by an edge to the boundary of each face, we obtain a locally finite infinite triangulation of the plane. Then the graph can be properly embedded into the plane, i.e. with no accumulation points in $\RR^2$ follows from Lemma 4.1 of \cite{Nach}. Then Theorem \ref{t613} follows from Corollary \ref{c69} and Theorem \ref{l613}.
$\hfill\Box$

\section{Critical Percolation Probabilities on Matching Graph Pairs}

\begin{lemma}\label{l113}Let $G=(V,E)$ be an infinite, connected, one-ended, planar graph. Assume that $G$ can be properly embedded into $\HH^2$ and that 
 the minimal vertex degree is at least 7.
\begin{enumerate}
\item Let $\mathcal{A}_1^*$ be the event that there is a unique infinite 1-cluster in the i.i.d.~Bernoulli site percolation of $G_*$. If
\begin{align}
\mathbb{P}_{p_c^{site}(G_*)}(\mathcal{A}_1^*)<1,\label{cr1}
\end{align}
then
\begin{align}
    p_u^{site}(G_*)\geq 1-p_c^{site}(G)>\frac{1}{2}\label{cl11}
\end{align}
\item Let $\mathcal{A}_1$ be the event that there is a unique infinite 1-cluster in the i.i.d.~Bernoulli site percolation of $G$. If
\begin{align*}
\mathbb{P}_{p_c^{site}(G)}(\mathcal{A}_1)<1.
\end{align*}
then
\begin{align*}
    p_u^{site}(G)\geq 1-p_c^{site}(G_*)>\frac{1}{2}
\end{align*}
\end{enumerate}
\end{lemma}

\begin{proof}We only prove Part (1) here; Part (2) can be proved similarly.

For each $p<1-p_c^{site}(G)$, the following cases might occur
\begin{itemize}
    \item $p<p_c^{site}(G_*)$, then a.s.~there are no infinite 1-*-clusters in the i.i.d.~Bernoulli($p$) site percolation on $G$.
    \item At $p:=p_c^{site}(G_*)$, (\ref{cr1}) holds.
    \item By Lemma \ref{lt1}, when $p\in (p_c^{site}(G_*),1-p_c^{site}(G))$,  a.s.~infinite 1-*-clusters have infinitely many ends, with strictly positive probability there are at least two distinct infinite 1-*-clusters.
\end{itemize}
Then (\ref{cl11}) follows from (\ref{dpu}).
\end{proof}

\begin{lemma}\label{l114}
\begin{enumerate}[label=(\Alph*)]
\item Let $G=(V,E)$ satisfy assumptions in Lemma \ref{l113}(1). Then
\begin{enumerate}
\item
for each $p>p_u^{site}(G_*)$, a.s.~there exists a unique infinite 1-cluster in the i.i.d.~Bernoulli($p$) site percolation on $G_*$.
\item 
\begin{align*}
    p_u^{site}(G_*)=1-p_c^{site}(G)
\end{align*}
\end{enumerate}
\item Let $G=(V,E)$ satisfy assumptions in Lemma \ref{l113}(2). Then
\begin{enumerate}
\item
for each $p>p_u^{site}(G)$, a.s.~there exists a unique infinite 1-cluster in the i.i.d.~Bernoulli($p$) site percolation on $G$.
\item 
\begin{align*}
    p_u^{site}(G)=1-p_c^{site}(G_*)
\end{align*}
\end{enumerate}
\end{enumerate}
\end{lemma}

\begin{proof}We only prove Part (A) of the Lemma; Part (b) can be proved similarly.

We first prove Part (Aa) of the theorem. By (\ref{cl11}), for any $p\in [p_u^{site}(G_*),1)$, there is uniform percolation at level $p$. Then the conclusion follows from Proposition \ref{p62} and the definition of $p_u^{site}(G_*)$.

Now we prove Part (Ab) of the theorem.
First we show that $p_u^{site}(G_*)\leq 1-p_c^{site}(G)$. For each $p>1-p_c^{site}(G)$, $1-p<p_c^{site}(G)$. Then a.s.~in the i.i.d.~Bernoulli($p$) site percolation, there are no infinite 0-clusters. Since $p_c^{site}(G_*)\leq p_c^{site}(G)<\frac{1}{2}$ if $p>1-p_c^{site}(G)$, then $p>\frac{1}{2}>p_c^{site}(G_*)$; hence a.s~there exist infinite 1-*-clusters
in the i.i.d.~Bernoulli($p$) site percolation.

If $p<p_u^{site}(G_*)$, with strictly positive probability there are at least two infinite 1-*-clusters. Since the graph $G$ is one-ended, planar, simple and locally finite; if with strictly positive probability there are at least two infinite 1-*-clusters, then with strictly positive probability there exists an infinite 0-cluster. But this contradicts $1-p<p_c^{site}(G)$. Therefore for any $p>1-p_c^{site}(G)$, we must have $p\geq p_u^{site}(G_*)$; this implies $p_u^{site}(G_*)\leq 1-p_c^{site}(G)$.
Then Part (2) follows from (\ref{c11}).
\end{proof}

 \begin{lemma}\label{lma83}Let $G=(V,E)$ be an infinite, connected, planar graph properly embedded into $\RR^2$. Assume that 
 the minimal vertex degree of $G$ is at least 7. If $G$ is semi-transitive, then
 \begin{align*}
 \mathbb{P}_{p_c^{site}(G)}(\mathcal{A}_f)=\mathbb{P}_{p_c^{site}(G)}(\mathcal{A}_1)=0.
 \end{align*}
 \end{lemma}

 \begin{proof}Let $V_F$ be given as in Definition \ref{df61}. Then for any $p\in [0,1]$, and $x\in V$,
 \begin{align*}
 \mathbb{P}_p(x\leftrightarrow\infty)\geq \min_{y\in V_F}\mathbb{P}_p(y\leftrightarrow\infty)
 \end{align*}
 Note that the graph $G$ is locally finite. By Lemma \ref{l96}, for any fixed $v$
 \begin{align*}
\lim_{n\rightarrow\infty}\sup_{u:d_{G}(u,v)\geq n} \mathbb{P}_{p_c^{site}(G)}(u\leftrightarrow v)=0.
 \end{align*}
 However if $\mathbb{P}_{p_c^{site}(G)}(\mathcal{A}_f)>0$, then $\mathbb{P}_{p_c^{site}(G)}(\mathcal{A}_1)>0$, by Lemma \ref{la68} we have
 \begin{align*}
 \mathbb{P}_{p_c^{site}(G)}(u\leftrightarrow v)\geq [\min_{y\in V_F}\mathbb{P}_p(y\leftrightarrow\infty)]^2\mathbb{P}_p(\mathcal{A}_1)>0
 \end{align*}
 for all $u,v$. The contradiction implies the Lemma.
 \end{proof}

\noindent{\textbf{Proof of Theorem \ref{mt1}}.} Theorem \ref{mt1} follows from Lemmas \ref{l114}(B) and \ref{lma83}.

\section{Vertex isoperimetric constant and Non-uniqueness at $\frac{1}{2}$}\label{sect:vic}

In this section, we prove Theorem \ref{l62}. Recall the following result proved in \cite{BSte}.

\begin{lemma}\label{lte}Let $G$ be a locally  finite graph with Cheeger constant $\mathbf{i}_{V}(G)\geq n$, where $n\geq 1$ is an integer.  Then $G$ has a spanning forest such that every tree in the forest is isomorphic to $T_{n+1}$. Here $T_{n+1}$ is the tree in which the root has degree $n$,  and the other vertices have degree $n+2$. 
\end{lemma}

Note that $\mathbf{i}_V(T_{n+1})=n$.

\bigskip

\noindent{\textbf{Proof of Theorem \ref{l62}}(1).} By Lemma \ref{lte}, $G$ has a spanning forest such that every tree in the forest is isomorphic to $T_{n+1}$, which is a tree whose root vertex has degree $n$, and all the other vertices have degree $n+2$. 

Fix a tree $T$ in the spanning forest.
Then $p_c^{site}(T)=\frac{1}{n+1}$. For any $p\in \left(\frac{1}{n+1},\frac{n}{n+1}\right)$ there are infinitely many infinite open clusters and infinitely many infinite closed clusters in the i.i.d.~Bernoulli$\left(p\right)$ site percolation on $T$.  Pick a vertex $w$ of $T$. Let $w0$ and $w1$ be two offsprings of $w$; and let $w00$, $w01$ (resp.\ $w10$, $w11$) are two offsprings of $w0$ (resp.\ $w1$). For $i,j\in\{0,1\}$, let $l_{ij}$ be the unique path of $T$ (consisting of edges of $T$) joining $w$ and $wij$. Assume that when the graph is identified with it embedding in the plane, the paths $l_{00}$, $l_{01}$, $l_{10}$ and $l_{11}$ are in cyclic order around $w$.

Suppose that the vertices $w;w0;w00;w1;w10$ are closed, and each of $w00;w10$ percolates (in closed vertices)in the subtree of $T$ rooted at it.  This implies that the open clusters intersecting the subtree rooted at $w01$ will be disjoint from those at the subtree rooted at $w11$, which gives non-uniqueness, because each of these subtrees is sure to contain infinite open clusters.  With probability 1.  Hence, by the Borel-Contelli lemma, for $p\in (\frac{1}{n+1};\frac{n}{n+1})$, a.s.~there are infinitely many infinite open clusters in the i.i.d.~Bernoulli($p$) site percolation on $G$.
$\hfill\Box$

\bigskip

\noindent\textbf{Proof of Theorem \ref{l62}(2).} By Lemma \ref{lte}, $G$ has a spanning forest such that every tree in the forest is isomorphic to $T_{2}$, which is a tree whose root vertex has degree $1$, and all the other vertices have degree $3$. Then $p_c^{site}(T)=\frac{1}{2}$. 

Since $G$ is planar, $G$ can be drawn on the plane in such a way that edges intersects only at vertices; and any compact subset of the plane intersects at most finitely many edges and vertices. Fix one tree $T$ in the spanning forest of $G$. For a nonnegative integer $k\geq 0$, let level-$k$ vertices consist of all the vertices of $T$ whose graph distance to the root is $k$.

Let $k\geq 1$ and $w$ be a level-$k$ vertex of $G$. Let $T_{w}$ be the subtree of $T$ rooted at $w$.

Passing through $w$, there exist a doubly infinite self-avoiding path $l_{w}$ consisting of edges of $T_w$, which divide the hyperbolic plane $\HH^2$ into two half spaces $H_1$ and $H_2$, such that
$H_2\cap T_{w}=\emptyset$. Let $G_{w}:=G\cap [H_1\cup l_{w}]$. Under the assumption that $G$ has uniformly bounded vertex degree and uniformly bounded face degree, we have
\begin{align}
p_c^{site}(G_{w})<\frac{1}{2};\ \forall\ w\in V(T). \label{pcgw}
\end{align}
We shall prove (\ref{pcgw}) in the appendix using enhancement arguments.

Since $T_w$ is a binary tree we can label the offsprings of each vertex by either 0 or 1.
Consider the i.i.d.~Bernoulli($\frac{1}{2}$) site percolation on $G_w$.
Suppose that the vertices $w;w0;w00;w1;w10$ are closed, and each of $w00;w10$ percolates (in closed vertices) in the subgraph $G_{w00};G_{w10}$ respectively.  This implies that the open clusters intersecting the subgraph $G_{w01}$ will be disjoint from those open clusters intersecting $G_{w11}$, which gives non-uniqueness, because each of these subgraphs is sure to contain infinite open clusters.  With probability 1, there is some such $w$.  Hence a.s.~there are infinitely many infinite open clusters in the i.i.d.~Bernoulli($\frac{1}{2}$) site percolation on $G_w$ by the Borel-Contelli lemma.
$\hfill\Box$

\appendix

\section{Proof of (\ref{pcgw})}

Let $G,T, T_w, G_w$ be given as in the proof of Theorem \ref{l62}. Let $p,s\in(0,1)$. Since $p_c^{site}(T_w)=\frac{1}{2}$, it suffices to show that $p_c^{site}(G_w)<p_c^{site}(T_w)$. 

To each vertex of $G_w$ associate an i.i.d.~Bernoulli($p$) random variable. To each edge in $E(G_w)\setminus E(T_w)$, associate an i.i.d.~Bernoulli($s$) random variable. Let $(\omega,\eta)\in \{0,1\}^{V(G_w)}\times \{0,1\}^{E(G_w)\setminus E(T_w)}$. Let $x\in V(G_w)$ and $e\in E(G_w)\setminus E(T_w)$. We define $\omega^x$, $\omega_x$, $\eta^e$, $\eta_e$ by
\begin{align*}
    \omega^x(y)=\begin{cases}\omega(y)&\mathrm{if}\ y\neq x\\ 1&\mathrm{if}\ y=x\end{cases};\qquad
    \omega_x(y)=\begin{cases}\omega(y)&\mathrm{if}\ y\neq x\\ 0&\mathrm{if}\ y=x\end{cases};\\
    \eta^e(f)=\begin{cases}\eta(f)&\mathrm{if}\ f\neq e\\ 1&\mathrm{if}\ f=e\end{cases};\qquad
    \eta_e(f)=\begin{cases}\eta(f)&\mathrm{if}\ f\neq e\\ 0&\mathrm{if}\ f=e\end{cases}.
\end{align*}

A (finite or infinite) open path in $(\omega,\eta)$ is an alternating sequence of vertices and edges of $G_w$
\begin{align}
v_1, e_1,v_2,e_2,v_3,e_3,\ldots\label{dpv}
\end{align}
such that for each $i\geq 1$
\begin{itemize}
    \item $v_i\in V(G_w)$, $e_i\in E(G_w)$ and
    $v_i$ and $v_{i+1}$ are two endpoints of $e_i$; and
    \item $\omega(v_i)=1$; and
    \item if $e_i$ is an edge of $E(G_w)\setminus E(T_w)$, then $\eta(e_i)=1$.
\end{itemize}

Let $v$ be a vertex of $G_w$. Let $A_n(v)$ be the event that $v$ is joined to $\partial B_{G_w}(v,n)$ by an open path in $(\omega,\eta)$; let $\theta_n(v,p,s)$ be the probability of $A_n(v)$. Here $B_{G_w}(v,n)$ is the subgraph of $G_w$ induced by all the vertices of $G_w$ whose graph distance to $v$ is at most $n$.  By the Russo's formula we obtain
\begin{align*}
&\frac{\partial \theta_n(v,p,s)}{\partial p}:=\sum_{z\in V(G_w)}
\mathbb{P}_{p,s}(z\ \mathrm{is\ pivotal
\ for\ }A_n(v));\\
&\frac{\partial \theta_n(v,p,s)}{\partial s}:=\sum_{e\in E(G_w)\setminus E(T_w)}
\mathbb{P}_{p,s}(e\ \mathrm{is\ pivotal
\ for\ }A_n(v));
\end{align*}
where
\begin{itemize}
    \item $v$ is pivotal for $A_n(v)$, if $I_{A_n(v)}(\omega^v,\eta)=1$ and $I_{A_n(v)}(\omega_v,\eta)=0$.
    \item $e$ is pivotal for $A_n(v)$, if $I_{A_n(v)}(\omega,\eta^e)=1$ and $I_{A_n(v)}(\omega,\eta_e)=0$.
\end{itemize}
and $I_{A_n(v)}(\omega,\eta)$ is the indicator for the event $A_n(v)$.

Let $R$ (resp.\ $D$) be the maximal face degree (resp.\ vertex degree) in $G$. Under the assumption that $G$ has uniformly bounded vertex degree and uniformly bounded face degree, we obtain that $R<\infty$ and $D<\infty$.
Let $z\in V(G_w)$. Assume
\begin{align*}
n\geq d_{T_w}(v,z)+DR+1;
\end{align*}
where $d_{T_w}(v,z)$ is the graph distance between $v$ and $z$ on $T_w$.
Consider the event that $z$ is pivotal for $A_n(v)$. Then there exists an open path $l_{vz}$ joining $v$ and $\partial A_n(v)$ passing through $z$ if $\omega(z)=1$; if $\omega(z)=0$ such an open path does not exist. Assume $l_{vz}$ has the form (\ref{dpv}) with $v=v_1$. Without loss of generality, we assume that $l_{vz}$ satisfies the following conditions
\begin{enumerate}[label=(\alph*)]
    \item $v_i\neq v_j$ whenever $i\neq j$; and
    \item $d_{T_w}(v_i,v_j)=1$ if and only if $|i-j|=1$.
\end{enumerate}
(a) is called the self-avoiding condition; and (b) is called the non-self-touching condition.

The following cases might occur
\begin{enumerate}
    \item along $l_{vz}$, $z$ is incident to an edge $e$ in $E(G_w)\setminus E(T_w)$. Let $(\hat{\omega},\hat{\eta})\in \{0,1\}^{V(G_w)}\times\{0,1\}^{E(G_w)\setminus E(T_w)}$ be the configuration obtained from $(\omega,\eta)$ as follows:
    \begin{itemize}
        \item for each $\langle z,z_1 \rangle\in E(T_w)\setminus l_{vz}$, let $\hat{\omega}(z_1)=0$; and
        \item for each $\langle z,z_2 \rangle\in E(G_w)\setminus [E(T_w)\cup l_{vz}]$, let $\hat{\eta}(\langle z, z_2\rangle)=0$.
    \end{itemize}
    Then in the new configuration $(\hat{\omega},\hat{\eta})$, $e$ is a pivotal edge for the event $A_n(v)$.
    \item along $l_{vz}$, $z$ is incident to no edges in $E(G_w)\setminus E(T_w)$. Let $e_3=\langle z,z_3 \rangle$ and $e_4=\langle z,z_4 \rangle$ be the two incident edges of $z$ in $l_{vz}\cap E(T_w)$. We can find a sequence of faces $f_1,\ldots,f_k$ in $G_w$ satisfying the following conditions:
    \begin{itemize}
        \item $e_3\in f_1$ and $e_4\in f_k$;
        \item for each $1\leq i\leq k-1$, $f_i$ and $f_{i+1}$ share an edge, one of whose endpoints is $z$;
        \item moving along $l_{vz}$, all faces $f_1,\ldots,f_k$ are one one side of $l_{vz}$;
    \end{itemize}
    The outer boundary $\zeta$ of $\cup_{i=1}^k f_k$ is the connected component of $\partial[\cup_{i=1}^k f_k]$ that is incident to the unbounded component of $\HH^2\setminus [\cup_{i=1}^k f_k]$. We can find two vertices $a,b\in \zeta\cap V(G_w)\cap l_{vz}$, such that the portion of $\zeta_{ab}$ of $\zeta$ between $a$ and $b$ is disjoint from $l_{vz}$.
    
    We write $\zeta_{ab}$ in the form (\ref{dpv}) with $a=v_1$, 
    then after changing the path if necessary, we can obtain that $\zeta_{ab}$ satisfies the self-avoiding condition (a) and non-self-touching condition (b).
    
    Note that $\zeta_{ab}$ contains at least one edge $e$ in $E(G_w)\setminus E(T_w)$.
    
    Consider $l_{vz}$ to start with $v$ and ending at a  vertex in $\partial B_n(v)$ in $G_w$.
    Let $l_1$ be the portion of $l_{vz}$ between $v$ and $a$ and let $l_2$ be the portion of $l_{vz}$ after $b$, such that 
    \begin{align*}
        l_{1}\cap \zeta_{ab}=\{a\};\ l_2\cap \zeta_{ab}=\{b\};\ l_1\cap l_2=\emptyset;\ l_{vz}=l_1\cup l_2\cup \zeta_{ab}.
    \end{align*}
    We already have that each one of $l_1$, $l_2$ and $\zeta_{ab}$ satisfies conditions (a) and (b); and that a vertex in $l_1$ and a vertex in $l_2$ cannot come within distance 1 (distance in $T_w$) of each other if there distance along $l_{vz}$ is at least 2. Hence (a) or (b) can only be violated by either (i) a vertex along $l_1$ and a vertex along $\zeta_{ab}$; or (ii) a vertex along $l_2$ and a vertex along $\zeta_{ab}$.
    
    Let $B_{G_w}(z,DR)$ be the subgraph of $G_w$ induced by the set of vertices in $G_w$ whose graph distance in $G_w$ to $z$ is at most $DR$. Note that
    \begin{align*}
    l_{ab}\subseteq B_{G_w}\left(z,\left\lfloor\frac{DR}{2} \right\rfloor\right)
    \end{align*}
    and
    \begin{align*}
        D\geq 3;\ R\geq 3.
    \end{align*}
    Hence if (a) or (b) is violated as in case (i) or case (ii); we can change configurations only in $B_{G_w}$ to make sure both (a) and (b) are satisfied. More precisely, if there exist $x\in l_1$ and $y\in \zeta_{ab}$ such that $d_{T_w}(x,y)=1$; and $x$ and $y$ are not adjacent vertices along $l_{vz}$, then
    \begin{itemize}
        \item let $\langle x_1,x \rangle $ be the first edge incident to $x$ visited by $l_{vz}$ and let $\langle y,y_1 \rangle$ be the last edge incident to $y$ visited by $l_{vz}$. Let $\langle x_2,x\rangle$ (resp.\ $\langle y_2,x\rangle$) be an arbitrary incident edge of $x$ (resp.\ $y$) other than $\langle x_1,x \rangle$ (resp.\ $\langle y_1,y \rangle$); 
        \begin{itemize}
            \item If $\langle x_2,x \rangle\in E(G_w)\setminus E(T_w)$,
            (resp.\ $\langle y_2,y \rangle\in E(G_w)\setminus E(T_w)$), make $\langle x_2,x\rangle$ (resp.\ $\langle y_2,y \rangle$) closed in $\eta$;
            \item If $\langle x_2,x \rangle\in E(T_w)$,
            (resp.\ $\langle y_2,y \rangle\in E(T_w)$), make $x_2$ (resp.\ $y_2$) closed in $\omega$.
        \end{itemize}
    \end{itemize}
    All the other cases of violations of (a) or (b) can be treated similarly. After changing configurations, the path $l_{vz}$ still contains at least one edge $e$ in $[E(G_w)\setminus E(T_w)]\cap B_{G_w}(z,DR)$.

     Let $(\hat{\omega},\hat{\eta})\in \{0,1\}^{V(G_w)}\times\{0,1\}^{E(G_w)\setminus E(T_w)}$ be the configuration obtained from $(\omega,\eta)$ as follows:
     \begin{itemize}
         \item For each vertex $u\in B_{G_w}(z,DR)$, if $u\notin l_{vz}$, let $\hat{\omega}(u)=0$.
         \item Let $\hat{\eta}=\eta$.
     \end{itemize}
\end{enumerate}
Then in the configuration $(\hat{\omega},\hat{\eta})$, $e$ is pivotal.
Hence we obtain
\begin{align*}
    \frac{\partial \theta_n(v,p,s)}{\partial p}&\leq \left[\max\left\{\frac{p}{1-p},\frac{1-p}{p}\right\}\right]^{|V(B_{G_w}(z,DR))|}\left[ \max\left\{\frac{s}{1-s},\frac{1-s}{s}\right\}\right]^{^{|E(B_{G_w}(z,DR))|}}\\
    &(|V(B_{G_w}(z,DR))|+|E(B_{G_w}(z,DR))|)
     \frac{\partial \theta_n(v,p,s)}{\partial s}\\
     &=\nu(p,s)\frac{\partial \theta_n(v,p,s)}{\partial s}
\end{align*}

Note that site percolation configuration on $G_w$ corresponds to $s=1$; while the site percolation on $T_w$ corresponds to $s=0$. Then the conclusion follows from similar arguments as in Page 71 of \cite{grgP}. 

\section{Explicit Construction of Embedded Trees}\label{sect:B}

\subsection{Construction of an embedded tree for a graph with minimal vertex degree at least 7}

Let $v\in V$. Let $v_{0}$, $v_{1}$ be two vertices adjacent to $v$ in $G$ such that $v$, $v_{0}$, $v_{1}$ share a face. Starting from $v,v_0$ construct a walk
\begin{align*}
\pi_0:=v,v_{0},v_{00},v_{000},\ldots,
\end{align*}
Starting from $v,v_1$ construct a walk 
\begin{align*}
\pi_1:=v,v_1,v_{11},v_{111},\ldots,
\end{align*}
such that
\begin{itemize}
    \item moving along $v_0,v,v_1$ in order, the face shared by $v_0,v,v_1$ is on the right; and
    \item moving along the walk $\pi_0$ starting from $v$, at each vertex $v_{0^k}$ ($k\geq 1$), there are exactly 3 incident faces on the right of $\pi_0$; and
    \item moving along the walk $\pi_1$ starting from $v$, at each vertex $v_{1^k}$ ($k\geq 1$), there are exactly 3 incident faces on the left of $\pi_1$.
\end{itemize}
One can show that both $\pi_0$ and $\pi_1$ are infinite and self-avoiding.

Let
\begin{align*}
   \pi_{1,1}:=\pi_1\setminus \{v\}=v_1,v_{11},v_{111},\ldots
\end{align*}

There exists $v_{01}\in V$ such that
\begin{itemize}
    \item $v_{01}$ is adjacent to $v_0$; and
    \item $v_0,v_{00},v_{01}$ share a face on the left of the walk $\pi_0$.
\end{itemize}
Similarly, there exist $v_{10},v_{1,\frac{1}{2}}\in V$ such that
\begin{itemize}
    \item both $v_{10}$ and $v_{1,\frac{1}{2}}$ are adjacent to $v_1$; and
    \item $v_1,v_{1,\frac{1}{2}},v_{11}$ share a face on the right of the walk $\pi_1$; and
    \item $v_1,v_{10},v_{1,\frac{1}{2}}$ share a face; moving along $v_{10},v_1,v_{1,\frac{1}{2}}$ in order, the face is on the right.
\end{itemize}
Note that $v_{01}\neq v$ and $v_{10}\neq v$, $v_{1,\frac{1}{2}}\neq v$ since each vertex in $G$ has degree at least 7.

Starting from $v_0,v_{01}$, construct a walk
\begin{align*}
    \pi_{01}:=v_0,v_{01},v_{011},v_{0111},\ldots
\end{align*}
Starting from $v_1,v_{10}$, construct a walk
\begin{align*}
    \pi_{10}:=v_1,v_{10},v_{100},v_{1000},\ldots
\end{align*}
Assume that
\begin{itemize}
    \item  moving along $v_{00},v_{0},v_{01}$ in order, the face shared by $v_{00},v_{0},v_{01}$ is on the right; and
     \item  moving along $v_{1,\frac{1}{2}},v_{1},v_{11}$ in order, the face shared by these vertices is on the right; and
    \item moving along the walk $\pi_{01}$ starting from $v_0$, at each vertex $v_{01^k}$ ($k\geq 1$), there are exactly 3 incident faces on the left of $\pi_{01}$; and
    \item moving along the walk $\pi_{10}$ starting from $v$, at each vertex $v_{10^k}$ ($k\geq 1$), there are exactly 3 incident faces on the right of $\pi_{10}$.
\end{itemize}
One can show that both walks are infinite and self-avoiding. Furthermore, let
\begin{align*}
    \tilde{\pi}_{01}:=v,\pi_{01};\qquad 
    \tilde{\pi}_{10}:=v,\pi_{10}
\end{align*}
One can show that $\tilde{\pi}_{01}$ is self-avoiding.

Moreover, one can show that $\tilde{\pi}_{10}$ is self-avoiding and that $\pi_{01}$ and $\pi_{10}$ are disjoint.  We repeat the same construction with $(v_0,v,v_1)$ replaced by $(v_{00},v_0,v_{01})$.

Starting $v_1,v_{1,\frac{1}{2}}$, we construct a walk
\begin{align*}
\pi_{1,\frac{1}{2}} :=   v_{1},v_{1,\frac{1}{2}}, v_{1,\frac{1}{2},1},v_{1,\frac{1}{2},1,1},\ldots
\end{align*}
 such that
\begin{itemize}
    \item moving along the walk $\pi_{1,\frac{1}{2}}$ staring from $v_1$, at each vertex $\pi_{1,\frac{1}{2},1^k}$ ($k\geq 0$), there are exactly 3 incident faces on the left.
\end{itemize}
Let
\begin{align*}
    &\tilde{\pi}_{1,\frac{1}{2}}:=v,\pi_{1,\frac{1}{2}}\\
    &\pi_{1,\frac{1}{2},1}:=\pi_{1,\frac{1}{2}}\setminus\{v_1\}=v_{1,\frac{1}{2}},v_{2,\frac{1}{2},1},v_{2,\frac{1}{2},1,1},\ldots
\end{align*}
One can show that $\tilde{\pi}_{1,\frac{1}{2}}$ is infinite and self-avoiding; and that
\begin{enumerate}[label=(\Alph*)]
\item The intersection of any two paths in $\pi_1,\pi_{10},\pi_{1,\frac{1}{2}}$ is $\{v_1\}$.
\item $\pi_{1,\frac{1}{2}}\cap \pi_0=\emptyset$ and $\pi_{1,\frac{1}{2}}\cap \pi_{01}=\emptyset$
\end{enumerate}

Let $v$ be the level-0 vertex, $v_0,v_1$ be level-1 vertices, and $v_{00},v_{01},v_{10},v_{1,\frac{1}{2}},v_{11}$ be the level-2 vertices. In general For $k\geq 2$, define the set $S_k$ of level-$k$ vertices as follows
\begin{align}
S_k:=\left\{v_{b}: b=(b_1,\ldots,b_k)\in \left\{0,\frac{1}{2},1\right\}^k; \mathrm{if}\ b_{j}=\frac{1}{2},\ \mathrm{then}\ j\geq 2,\ \mathrm{and}\ b_{j-1}=1.\right\}.\label{dsk}
\end{align}
Assume we defined all the level-$k$ vertices. For each $v_b\in S_k$, the following cases might occur
\begin{itemize}
    \item $b_k=0$: in this case we define 2 paths $\pi_{b,0}$, $\pi_{b,1}$ as defining $\pi_0$ and $\pi_1$ with $v_b$ replaced by $v$.
    \item $b_k=1$: in this case we define 3 paths $\pi_{b,0}$,
    $\pi_{b,\frac{1}{2}}$ $\pi_{b,1}$ as defining $\pi_{10}$ $\pi_{1,\frac{1}{2}}$ and $\pi_{11}$ with $v_b$ replaced by $v_1$.
    \item $b_k=\frac{1}{2}$: in this case we define 2 paths $\pi_{b,0}$, $\pi_{b,1}$ as defining $\pi_{0}$ and $\pi_{1}$  with $v_b$ replaced by $v$.
\end{itemize}

Then we find a tree $T$ whose vertex set consists of $\{v,v_0,v_1\}\cup_{k\geq 2}S_k$ ane edge set consists of all the edges along a path $\pi_{b}$ such that for some $k\geq 1$ $b=(b_1,\ldots,b_k)\in \left\{0,\frac{1}{2},1\right\}^k; \mathrm{if}\ b_{j}=\frac{1}{2},\ \mathrm{then}\ j\geq 2,\ \mathrm{and}\ b_{j-1}=1$
as a subgraph of $G$.

\subsection{Construction of an embedded tree for a graph with minimal vertex degree at least 5 and face degree at least 4}

Let $v\in V$. Let $v_{0}$, $v_{1}$ be two vertices adjacent to $v$ in $G$ such that $v$, $v_{0}$, $v_{1}$ share a face. Starting from $v,v_0$ construct a walk
\begin{align*}
\pi_0:=v,v_{0},v_{00},v_{000},\ldots,
\end{align*}
Starting from $v,v_1$ construct a walk 
\begin{align*}
\pi_1:=v,v_1,v_{11},v_{111},\ldots,
\end{align*}
such that
\begin{itemize}
    \item moving along $v_0,v,v_1$ in order, the face shared by $v_0,v,v_1$ is on the right; and
    \item moving along the walk $\pi_0$ starting from $v$, at each vertex $v_{0^k}$ ($k\geq 1$), there are exactly 2 incident faces on the right of $\pi_0$; and
    \item moving along the walk $\pi_1$ starting from $v$, at each vertex $v_{1^k}$ ($k\geq 1$), there are exactly 2 incident faces on the left of $\pi_1$.
\end{itemize}
Then we can show that both $\pi_0$ and $\pi_1$ are infinite and self-avoiding.

Let
\begin{align*}
   \pi_{1,1}:=\pi_1\setminus \{v\}=v_1,v_{11},v_{111},\ldots
\end{align*}

There exists $v_{01}\in V$ such that
\begin{itemize}
    \item $v_{01}$ is adjacent to $v_0$; and
    \item $v_0,v_{00},v_{01}$ share a face on the left of the walk $\pi_0$.
\end{itemize}
Similarly, there exist $v_{10},v_{1,\frac{1}{2}}\in V$ such that
\begin{itemize}
    \item both $v_{10}$ and $v_{1,\frac{1}{2}}$ are adjacent to $v_1$; and
    \item $v_1,v_{1,\frac{1}{2}},v_{11}$ share a face on the right of the walk $\pi_1$; and
    \item $v_1,v_{10},v_{1,\frac{1}{2}}$ share a face; moving along $v_{10},v_1,v_{1,\frac{1}{2}}$ in order, the face is on the right.
\end{itemize}
Note that $v_{10}\neq v$, $v_{01}\neq v$ and $v_{1,\frac{1}{2}}\neq v$ since each vertex in $G$ has degree at least 5.

Starting from $v_0,v_{01}$, construct a walk
\begin{align*}
    \pi_{01}:=v_0,v_{01},v_{011},v_{0111},\ldots
\end{align*}
Starting from $v_1,v_{10}$, construct a walk
\begin{align*}
    \pi_{10}:=v_1,v_{10},v_{100},v_{1000},\ldots
\end{align*}
such that
\begin{itemize}
    \item  moving along $v_{00},v_{0},v_{01}$ in order, the face shared by $v_{00},v_{0},v_{01}$ is on the right; and
     \item  moving along $v_{1,\frac{1}{2}},v_{1},v_{11}$ in order, the face shared by these vertices is on the right; and
    \item moving along the walk $\pi_{01}$ starting from $v_0$, at each vertex $v_{01^k}$ ($k\geq 1$), there are exactly 2 incident faces on the left of $\pi_{01}$; and
    \item moving along the walk $\pi_{10}$ starting from $v$, at each vertex $v_{10^k}$ ($k\geq 1$), there are exactly 2 incident faces on the right of $\pi_{10}$.
\end{itemize}
Then we can show that both walks are infinite and self-avoiding. Furthermore, let
\begin{align*}
    \tilde{\pi}_{01}:=v,\pi_{01};\qquad 
    \tilde{\pi}_{10}:=v,\pi_{10}
\end{align*} 
We can show that $\tilde{\pi}_{01}$ and $\tilde{\pi}_{10}$ are self-avoiding, and that $\pi_{01}$ and $\pi_{10}$ never intersect each other.

Starting $v_1,v_{1,\frac{1}{2}}$, we construct a walk
\begin{align*}
\pi_{1,\frac{1}{2}} :=   v_{1},v_{1,\frac{1}{2}}, v_{1,\frac{1}{2},1},v_{1,\frac{1}{2},1,1},\ldots
\end{align*}
 such that
\begin{itemize}
    \item moving along the walk $\pi_{1,\frac{1}{2}}$ staring from $v_1$, at each vertex $\pi_{1,\frac{1}{2},1^k}$ ($k\geq 0$), there are exactly 2 incident faces on the left.
\end{itemize}
Let
\begin{align*}
    &\tilde{\pi}_{1,\frac{1}{2}}:=v,\pi_{1,\frac{1}{2}}\\
    &\pi_{1,\frac{1}{2},1}:=\pi_{1,\frac{1}{2}}\setminus\{v_1\}=v_{1,\frac{1}{2}},v_{2,\frac{1}{2},1},v_{2,\frac{1}{2},1,1},\ldots
\end{align*}
We can show that $\tilde{\pi}_{1,\frac{1}{2}}$ is infinite and self-avoiding. Moreover, one can prove that
\begin{enumerate}[label=(\Alph*)]
\item The intersection of any two paths in $\pi_1,\pi_{10},\pi_{1,\frac{1}{2}}$ is $\{v_1\}$.
\item $\pi_{1,\frac{1}{2}}\cap \pi_0=\emptyset$ and $\pi_{1,\frac{1}{2}}\cap \pi_{01}=\emptyset$
\end{enumerate}

Let $v$ be the level-0 vertex, $v_0,v_1$ be level-1 vertices, and $v_{00},v_{01},v_{10},v_{1,\frac{1}{2}},v_{11}$ be the level-2 vertices. In general for $k\geq 2$, define the set $S_k$ of level-$k$ vertices as in (\ref{dsk}).

Assume we find all the level-$k$ vertices. For each $v_b\in S_k$, the following cases might occur
\begin{itemize}
    \item $b_k=0$: in this case we define 2 paths $\pi_{b,0}$, $\pi_{b,1}$ as defining $\pi_0$ and $\pi_1$ with $v_b$ replaced by $v$.
    \item $b_k=1$: in this case we define 3 paths $\pi_{b,0}$,
    $\pi_{b,\frac{1}{2}}$ $\pi_{b,1}$ as defining $\pi_{10}$ $\pi_{1,\frac{1}{2}}$ and $\pi_{11}$ with $v_b$ replaced by $v_1$.
    \item $b_k=\frac{1}{2}$: in this case we define 2 paths $\pi_{b,0}$, $\pi_{b,1}$ as defining $\pi_{0}$ and $\pi_{1}$  with $v_b$ replaced by $v$.
\end{itemize}

Then we find a tree $T$ whose vertex set consists of $\{v,v_0,v_1\}\cup_{k\geq 2}S_k$ and edge set consists of all the edges along a path $\pi_{b}$ such that for some $k\geq 1$ $b=(b_1,\ldots,b_k)\in \left\{0,\frac{1}{2},1\right\}^k; \mathrm{if}\ b_{j}=\frac{1}{2},\ \mathrm{then}\ j\geq 2,\ \mathrm{and}\ b_{j-1}=1$
as a subgraph of $G$.

\bigskip
\noindent\textbf{Acknowledgements.}\ ZL acknowledges support from National Science Foundation DMS 1608896 and Simons Foundation grant 638143. 

\bibliography{psg}
\bibliographystyle{plain}
\end{document}